\documentclass[twoside,a4paper,leqno,12pt]{amsproc}
\usepackage[top=30mm,right=30mm,bottom=30mm,left=30mm]{geometry}

\usepackage[pagebackref]{hyperref}
\hypersetup{citecolor=blue, linkcolor=blue, colorlinks=true}
\usepackage{amsmath, amssymb, amsthm, amsfonts, color, amsrefs}
\usepackage{graphicx, float, array, multicol, rotating}
\usepackage{mathtools,gensymb}
\usepackage{pgfplotstable} 

\renewcommand{\leq}{\leqslant}
\renewcommand{\geq}{\geqslant}

\newcommand{\Z}{\mathbb{Z}}

\renewcommand{\le}{\leqslant}
\renewcommand{\ge}{\geqslant}
\renewcommand{\leq}{\leqslant}
\renewcommand{\geq}{\geqslant}

\usepackage[shortlabels]{enumitem}
\setlist[enumerate]{label=\rm{(\alph*)}}

\theoremstyle{definition}
\newtheorem{definition}{Definition}

\theoremstyle{plain}

\newtheorem{theorem}[definition]{Theorem}

\newtheorem{lemma}[definition]{Lemma}
\newtheorem{corollary}[definition]{Corollary}

\begin{document}

\author{Alice Devillers}
\address{\noindent Alice Devillers and Stephen Glasby, Centre for the Mathematics of Symmetry and Computation, University of Western Australia, 35 Stirling Highway, Perth 6009, Australia. \emph{E-mail addresses:} {\tt\texttt{\{alice.devillers, stephen.glasby\}@uwa.edu.au}}
}

\author{S.\,P. Glasby}
 
\thanks{Acknowledgements: SG gratefully acknowledges support from the Australian Research Council Discovery Project DP190100450.\newline 
2010 Math Subject Classification: 11A07, 11-01.
}
 
\title[Roots of unity and unreasonable differentiation]{Roots of unity and unreasonable differentiation}

\subjclass[2010]{}
\date{\today}

\begin{abstract}
  We explore when
  it is legal to differentiate a polynomial evaluated at a root of unity using modular arithmetic.
\end{abstract}

\maketitle

\section{Sometimes legal operations}\label{s:legal}

  The equation $(x+y)^p=x^p+y^p$ is valid in a field of prime
  characteristic~$p$. Thus an apparent error can be a legitimate deduction
  in the right circumstances.



 Denote the $k$th derivative of 
$t^n$ as $n^{\underline{k}}t^{n-k}$ where $n^{\underline{k}}$ equals
 $n(n-1)\cdots(n-k+1)$ for $k>0$, and 1 if $k=0$. We pronounce $n^{\underline{k}}$
 as ``$n$ falling factorial $k$''. Observe that
 $n^{\underline{k}}$ is divisible by each integer in $\{1,\dots,k\}$, and $n^{\underline{k}}=0$ for
 $n<k$. Also the $k$th derivative $f^{(k)}(t)$ of a power series $f(t)=\sum_{n\ge0}f_nt^n$ equals $\sum_{n\ge0}f_nn^{\underline{k}}t^{n-k}=\sum_{n\ge k}f_nn^{\underline{k}}t^{n-k}$.

 Let $\alpha\in\Z$ satisfy $\alpha^n\equiv1\pmod n$. Thus $\alpha$ is a root, modulo~$n$, of the polynomial $t^n-1=(t-1)f(t)$, where $f(t)=\sum_{i=0}^{n-1} t^i$.
 It is clear that $f^{(0)}(\alpha)  \equiv 0 \pmod n$ when
 $\alpha\equiv 1 \pmod n$. However, it seems unreasonable to expect that
 $f^{(k)}(\alpha)= \sum_{i=0}^{n-1}i^{\underline{k}}\alpha^{i-k}\equiv 0 \pmod n$ holds for all $k\ge0$. What looks like a blunder turns out to be true under the
 (unreasonably) weak assumptions of Theorem~\ref{T1}.


\begin{theorem}\label{T1}
  Suppose $k\ge0$, $n\ge1$, $\alpha\in\Z$ where $\alpha^n\equiv1\pmod n$.
  Then
  \begin{equation}\label{E3}
   f^{(k)}(\alpha)= \sum_{i=0}^{n-1}i^{\underline{k}}\alpha^{i-k}\equiv 0 \pmod n.
  \end{equation}
  if and only if at least one of the following hold:
  \begin{enumerate}[{\rm (a)}]
    \item  $k+1\not\in\{4,q\}$ where $q$ is prime, or
    \item $k+1=4$ and $4\nmid n$, or
    \item  $k+1$ is a prime $q$, and $q\nmid n$ or $\alpha\not\equiv1 \pmod q$.
  \end{enumerate}
\end{theorem}

The motivation for Theorem~\ref{T1} came from a (presently
unfinished~\cite{BDG}) study of input-output
automata on a group $G$. We considered
the finite groups $G$ for which there exists a `constant' $k\in G$ and
a function $f\colon G\to G$ satisfying $f(xk)=xf(x)$ for all $x\in G$.
We call these $J$-groups (as they are related to the Jacobson radical of
a near ring).
A simple argument shows that $J$-groups must have odd order, and hence
are solvable by~\cite{FT}. We conjectured~\cite{BDG} that any nilpotent group of
odd order is a $J$-group. To prove that many metacyclic groups
are $J$-groups required the $k=0$ and $k=1$ cases of Theorem~\ref{T1}. The proof for
all $k\ge0$ is not much harder.

\section{The proofs}\label{s:proofs}

We first establish some preliminary results before proving Theorem~\ref{T1}.

Henceforth, $n,i,j,k$ will be integers. 

A sum $\sum_{i=n_0}^{n_1-1}g(i)$ collapses if we find a function
$G$ such that $g(i)=G(i+1)-G(i)$ for $n_0\leq i<n_1$. Then
$\sum_{i=n_0}^{n_1-1}g(i)=G(n_1)-G(n_0)$. By analogy with differentiation, we write
$(\Delta G)(i)=G(i+1)-G(i)$. For example, if $g(i)=i^{\underline{k}}$, then it
follows from $\Delta(i^{\underline{k+1}})=(i+1)i^{\underline{k}}-i^{\underline{k}}(i-k)
  =(k+1)i^{\underline{k}}$ that $G(i)=i^{\underline{k+1}}/(k+1)$. Hence
\begin{equation}\label{E1}
  \sum_{i=n_0}^{n_1-1}i^{\underline{k}}
  =\sum_{i=n_0}^{n_1-1}\Delta\left(\frac{i^{\underline{k+1}}}{k+1}\right)=
  \frac{n_1^{\underline{k+1}}}{k+1}-\frac{n_0^{\underline{k+1}}}{k+1}
  =\frac{n_1^{\underline{k+1}}-n_0^{\underline{k+1}}}{k+1}.
\end{equation}

Clearly $k$ divides $n^{\underline{k}}$ for all $n\ge0$ and $k\ge1$.
The $p$-adic valuation $\nu_p(n)$ of an integer $n\ne 0$ is defined by 
$\nu_p(n)=\log_p(n_p)$ where $n_p$ is the largest $p$-power divisor of $n$.
This (additive) valuation extends to $\mathbb{Q}^\times$ by defining $\nu_p(r/s)$ to be  $\nu_p(r)-\nu_p(s)$.

\begin{lemma}\label{L1}
  Suppose $k\ge1$ and $n\ge1$. Let $p\mid (k+1)$ where $p$ is a prime, and let $e=\nu_p(k+1)\ge 1$.
  \begin{enumerate}[{\rm (a)}]
  \item If $k+1\ne p^e$, then
    $\nu_p((n-1)^{\underline{k}})\ge e$.
    \item If $k+1=p^e$, then $\nu_p((n-1)^{\underline{k}})\ge e-1$ where equality holds only if $p\mid n$.
  \item  $\nu_p((n-1)^{\underline{k}}/(k+1))<0$ if and only if $k+1\in\{4,p\}$
    and $(k+1)\mid n$, in which case $\nu_p((n-1)^{\underline{k}}/(k+1))=-1$.
  \end{enumerate}
\end{lemma}

\begin{proof} 
(a)~Suppose first that
$k+1$ is not a $p$-power and write $k+1=ab$ where $\gcd(a,b)=1$ and
$1<a<b<k+1$. Since $a,b\le k$ it follows that $a$ and $b$, and hence
  $k+1=ab$, divide $(n-1)^{\underline{k}}$.
  Hence $e=\nu_p(k+1)\le\nu_p((n-1)^{\underline{k}})$. This proves~(a).
  
(b)~Suppose now that $k+1=p^e$. As $p^{e-1}\leq k$, we deduce that
$p^{e-1}\mid (n-1)^{\underline{k}}$, and so $\nu_p((n-1)^{\underline{k}})\ge e-1$. Suppose $\nu_p((n-1)^{\underline{k}})=e-1$. As $k+1$ divides $n^{\underline{k+1}}=n(n-1)^{\underline{k}}$
but not $(n-1)^{\underline{k}}$, we deduce that $p$ divides $n$.
 This proves~(b).

(c)~Assume first that $\nu_p((n-1)^{\underline{k}}/(k+1))<0$, that is, $\nu_p((n-1)^{\underline{k}})<\nu_p(k+1)=e$. Part~(a) implies $k+1=p^e$ and Part~(b) implies $\nu_p((n-1)^{\underline{k}})= e-1$ and $p\mid n$,  so that $\nu_p((n-1)^{\underline{k}}/(k+1))=-1$. 
Thus each factor of $(n-1)^{\underline{k}}$ of the form $n-jp$ with
$1\le j\le p^{e-1}-1$ is a multiple of $p$, and so $\nu_p((n-1)^{\underline{k}})\geq p^{e-1}-1$. Therefore
$p^{e-1}-1\le e-1$, that is $p^{e-1}\le e$.
The latter inequality is true
for $e=1$ and all primes $p$, and for $e=2$ and $p=2$, and false otherwise.

If $e=1$, then $k+1=p\mid n$. If $e=2$ and $p=2$, then $k+1=4$, $2\mid n$ and $\nu_2((n-1)^{\underline{3}})= 1$. Thus $n-1$ and $n-3$ are odd while $n-2\equiv 2\pmod 4$. It follows that $n\equiv 0\pmod 4$, and so in both cases $k+1$ divides $n$.

Conversely, assume that $k+1\in\{4,p\}$ and $(k+1)\mid n$.
If $k+1=4\mid n$, then $(n-1)^{\underline{k}}=(n-1)(n-2)(n-3)$ where $n-i\equiv 4-i\pmod 4$. Thus $\nu_2((n-1)^{\underline{3}})= 1<\nu_2(k+1)=2$.
If $k+1=p\mid n$, then $(n-1)^{\underline{k}}=(n-1)(n-2)\cdots(n-p+1)$ where $n-k\equiv p-k \not\equiv 0\pmod p$. Thus, in both cases, we have { $\nu_p((n-1)^{\underline{k}})= e-1<\nu_p(k+1)=e$, as desired.}
\end{proof}

\begin{corollary}\label{C}
  By Lemma~$\ref{L1}$, we have that $(n-1)^{\underline{k}}/(k+1)$ is an integer unless 
  \begin{enumerate}[{\rm (i)}]
    \item $k+1=4$ and $4\mid n$, or
    \item $k+1$ is a prime $p$, and $p\mid n$.
  \end{enumerate}
  Moreover, $\nu_p((n-1)^{\underline{k}}/(k+1))\ge-1$ and
  $\nu_p((n-1)^{\underline{k}}/(k+1))\ge0$ if $(k+1)\nmid n$.
\end{corollary}


To prove Theorem~\ref{T1}, we will also need the following lemma.

\begin{lemma}\label{L2}
  Suppose $k\geq0$, $n\ge1$ and $\alpha\equiv1\pmod p$ where $p$ is prime. Then
  \[
  \sum_{i=0}^{n-1}i^{\underline{k}}\alpha^{i-k}\equiv \frac{(n-1)^{\underline{k}}}{k+1}n\pmod {p^\ell}\qquad\textup{where $\ell:=\nu_p(n)$.}
  \]
\end{lemma}
\begin{proof} 
  Since $\alpha\equiv1\pmod p$, we have $\alpha=1+y$ where $p\mid y$. Using $i^{\underline{k}}=0$ if $0\le i<k$,
  and Eq.~\eqref{E1} gives:
\begin{align}\label{E4}
  \sum_{i=0}^{n-1}i^{\underline{k}}\alpha^{i-k}&=
  \sum_{i=k}^{n-1}i^{\underline{k}}(1+y)^{i-k}=
  \sum_{i=k}^{n-1}i^{\underline{k}}\sum_{j=0}^{i-k}\binom{i-k}{j}y^j\notag\\ 
  &=\sum_{j=0}^{n-1-k}y^j\sum_{i=k+j}^{n-1}i^{\underline{k}}\binom{i-k}{j}=\sum_{j=0}^{n-1-k}\frac{y^j}{j!}\sum_{i=k+j}^{n-1}i^{\underline{k+j}}\notag\\
 &   =\sum_{j=0}^{n-1-k}\frac{y^j}{j!}\left(\frac{n^{\underline{k+j+1}}}{k+j+1}
  -\frac{(k+j)^{\underline{k+j+1}}}{k+j+1}\right) \notag\\
 & =\sum_{j=0}^{n-1-k}\frac{y^j}{j!}\frac{n^{\underline{k+j+1}}}{k+j+1} =\sum_{j=0}^{n-1-k}\frac{y^j}{j!}\frac{(n-1)^{\underline{k+j}}}{k+j+1}n.
\end{align}


Consider the summands in~\eqref{E4} with $j\ge 1$.
By Legendre's formula, $\nu_p(j!)=\sum_{i=1}^\infty \lfloor\frac{j}{p^i}\rfloor$.
(The sum is finite as $\lfloor\frac{j}{p^i}\rfloor=0$ for $p^i>j$.)
Thus $\nu_p(j!)<\sum_{i=1}^\infty \frac{j}{p^i}=\frac{j}{p-1}$. Therefore
 $\nu_p(j!)< j$ for $j\ge1$, and hence $p$ divides $y^j/j!$.
Since $\nu_p(y^j/j!)\ge 1$ for all $j\ge1$, and
$\nu_p(\frac{(n-1)^{\underline{k+j}}}{k+j+1})\ge -1$ by Corollary~\ref{C}, we have
$\nu_p(\frac{y^j}{j!}\frac{(n-1)^{\underline{k+j}}}{k+j+1})\ge 0$. However $p^\ell\mid n$, and so $\frac{y^j}{j!}\frac{(n-1)^{\underline{k+j}}}{k+j+1}n\equiv 0\pmod {p^\ell}$ for each $j\geq 1$.
Hence
 \[
   \sum_{i=0}^{n-1}i^{\underline{k}}\alpha^{i-k}\equiv \frac{(n-1)^{\underline{k}}}{k+1}n\pmod {p^\ell}.\qedhere
\]
\end{proof}

\medskip 

\begin{proof}[Proof of Theorem~\ref{T1}]
  When $n=1$, Eq.~\eqref{E3} is trivially true.
  Moreover, one of (a), (b) or (c) is true when $n=1$ since $4\nmid n$ and $q\nmid n$ for any prime $q$.

We now assume $n>1$.
  Suppose that $n=n_1\cdots n_r$ where the $n_j$ are pairwise coprime
  and $n_j>1$ for each $j$.
  Given the ring isomorphism $\Z_n\to\Z_{n_1}\times\cdots\times\Z_{n_r}$, Eq.~\eqref{E3} holds if and only if $n_j\mid\sum_{i=0}^{n-1}i^{\underline{k}}\alpha^{i-k}$
  for each $j$. Suppose that $n_j=p_j^{\ell_j}$ where each $p_j$ is prime.
  Fix a prime factor $p$ of $n$, and set $\ell:=\nu_p(n)$.
  
  
  For each prime factor $p$ of $n$, we divide the proof in two cases.
  
  {\sc Claim~1:} If $\alpha\not\equiv1\pmod p$, then  $\sum_{i=0}^{n-1}i^{\underline{k}}\alpha^{i-k}\equiv0\pmod {p^\ell}$.
  
  Suppose that $\alpha\not\equiv1\pmod p$.
  Consider the identity $f(t)=\sum_{i=0}^{n-1}t^i=f_1(t)f_2(t)$ where
  $f_1(t)=t^n-1$ and $f_2(t)=(t-1)^{-1}$. The $k$-fold derivative of the product
  $f_1f_2$ is 
  $(f_1f_2)^{(k)}=\sum_{i=0}^k\binom{k}{i}f_1^{(k-i)}f_2^{(i)}$ by Leibnitz' formula. We have
  $f_1^{(i)}(t)=n^{\underline{i}}t^{n-i}$ for $i>0$, and
  $f_2^{(i)}(t)=(-1)^ii!(t-1)^{-1-i}=-i!(1-t)^{-1-i}$ for $i\ge0$. Hence, for $t\neq 1$,
  \[
f^{(k)}(t)=\sum_{i=0}^{n-1}i^{\underline{k}}t^{i-k}
  =f_1^{(0)}(t)f_2^{(k)}(t)-\sum_{i=0}^{k-1}\binom{k}{i}n^{\underline{k-i}}t^{n-k+i}i!(1-t)^{-1-i}.
  \]
  Replacing $\binom{k}{i}i!$ with $k^{\underline{i}}$ gives
  \[
    \sum_{i=0}^{n-1}i^{\underline{k}}t^{i-k}
    =-(t^n-1)k!(1-t)^{-1-k}-\sum_{i=0}^{k-1}k^{\underline{i}}n^{\underline{k-i}}t^{n-k+i}(1-t)^{-1-i}.
  \]
  Substituting $t=\alpha$ and noting that $\alpha^n\equiv1\pmod {p^\ell}$
  and $1-\alpha$ is a unit modulo~$p^\ell$ shows that
  $\sum_{i=0}^{n-1}i^{\underline{k}}\alpha^{i-k}\equiv0\pmod {p^\ell}$.
  The sum vanishes modulo $p^\ell$ for $k=0$, and for $k>0$
  because $n$ divides $n^{\underline{k-i}}$ for $i<k$ and $p^\ell\mid n$.
  This proves Claim~1.

{\sc Claim~2:} If $\alpha\equiv1\pmod p$ and at least one of the conditions (a), (b), or~(c) hold, then 
  $\sum_{i=0}^{n-1}i^{\underline{k}}\alpha^{i-k}\equiv 0\pmod {p^\ell}$. 

  Suppose that $\alpha\equiv1\pmod p$  and  at least one of the conditions (a), (b), or~(c) hold.
  We argue that $\frac{(n-1)^{\underline{k}}}{k+1}$ is an integer. Suppose not.
  Then by Corollary~\ref{C}, condition~(i) or (ii) holds.
  Since $\alpha\equiv1\pmod p$ both ~(i) and (ii) are incompatible with
  (a), (b), and~(c). This shows that $\frac{(n-1)^{\underline{k}}}{k+1}$ is
  an integer, and hence Claim~2 is true by Lemma~\ref{L2}.
%

In summary, we have $n_j\mid\sum_{i=0}^{n-1}i^{\underline{k}}\alpha^{i-k}$ for each $j$, and so~\eqref{E3} holds if  at least one of the conditions (a), (b), or (c) hold.
  
To finish the proof, we now prove the converse. 
Assume~\eqref{E3} holds but (a) does not hold. Then $n_j\mid\sum_{i=0}^{n-1}i^{\underline{k}}\alpha^{i-k}$
  for each $j$, and $k+1\in\{4,q\}$ where $q$ is a prime. We must show that conditions~(b) or~(c) hold.

  Suppose $k+1=4$. Assume $4\mid n$. Consider the prime factor $2$ of $n$.
  Since $\alpha^n\equiv1\pmod n$, we have $\alpha\equiv1\pmod 2$. Thus Lemma~\ref{L2} implies
  $\sum_{i=0}^{n-1}i^{\underline{3}}\alpha^{i-3}\equiv \frac{(n-1)^{\underline{3}}}{4}n\pmod {2^\ell}$ where $\nu_2(n)=\ell$. 
  By Lemma~\ref{L1}(c),  $\nu_2((n-1)^{\underline{3}}/4)<0$, and so $2^\ell$ does not divide $\frac{(n-1)^{\underline{3}}}{4}n$, a contradiction. Hence, if $k+1=4$,
  then $4\nmid n$ and (b) holds.

  Suppose $k+1$ is a prime $q$ and $\alpha\equiv 1 \pmod q$. Assume $q\mid n$.
  Consider the prime factor $q$ of $n$. Then
$\sum_{i=0}^{n-1}i^{\underline{k}}\alpha^{i-k}\equiv \frac{(n-1)^{\underline{k}}}{k+1}n\pmod {q^\ell}$  where $\nu_q(n)=\ell$, by Lemma~\ref{L2}. However, Lemma~\ref{L1}(c) shows  $\nu_q((n-1)^{\underline{k}}/(k+1))<0$, and so $q^\ell$ does not divide  $\frac{(n-1)^{\underline{k}}}{k+1}n$,  a contradiction.
Hence, if $k+1=q$ with $\alpha\equiv 1 \pmod q$, then $q\nmid n$ and condition (c) holds.
\end{proof}

Finally, observe that the requirement $\alpha\not\equiv 1\pmod q$
in Theorem~\ref{T1}(c) is needed. For example take $\alpha=5$,
$n=6$, $k=2$.
Then $\sum_{i=0}^{n-1}i^{\underline{k}}\alpha^{i-k}\equiv 0\pmod n$ and $q=k+1=3$.
Thus conditions~(a) and (b) do not hold, and the only part of~(c) that holds is
$\alpha\not\equiv1\pmod q$. 
In condition (b), we do not need to add ``or $\alpha\not\equiv 1 \pmod 2$'' because that never happens when $4\mid n$.

\vskip2mm{\sc Acknowledgement.} 
We thank the referee for their helpful suggestions.


\end{document}